\documentclass[12 pt]{amsart}

\newtheorem{theorem}{Theorem}[section]
\newtheorem{lemma}[theorem]{Lemma}

\theoremstyle{definition}

\theoremstyle{remark}

\usepackage{amscd,amssymb}

\begin{document}
\baselineskip 15.5pt

\title[Endomorphisms preserving coordinates] 
{Endomorphisms preserving coordinates of polynomial algebras}

\author[Yun-Chang Li and Jie-Tai Yu]
{Yun-Chang Li and Jie-Tai Yu}
\address{Department of Mathematics, The University of Hong
Kong, Hong Kong SAR, China} \email{liyunch@hku.hk,\ liyunch@163.com}
\address{Department of Mathematics, The University of Hong
Kong, Hong Kong SAR, China} \email{yujt@hku.hk,\
yujietai@yahoo.com}

\thanks
{The research of Yun-Chang Li  was partially supported by a
postgraduate studentship in the University of Hong Kong.}

\thanks{The research of Jie-Tai Yu was partially
supported by an RGC-GRF Grant.}



\maketitle

\noindent $\mathbf{Abstract.}$\  It is proved that the Jacobian
 of a $k$-endomorphism of $k[x_1,\dots,x_n]$
over a field $k$ of characteristic zero taking
every tame coordinate to a coordinate, must be a nonzero constant in $k$.
It is also proved that the Jacobian  of an $R$-endomorphism
of $A:=R[x_1,\dots,x_n]$ (where $R$ is a polynomial ring in finite number of variables
over an infinite field $k$), taking every $R$-linear coordinate of $A$
to an $R$-coordinate of $A$, is a nonzero constant in $k$.

\smallskip

\section{\bf Introduction and the main results}

\noindent Van den Essen and Shpilrain  \cite{ess} asked
the following

\vspace{5mm}

\noindent {\bf Question 1.} {\it Let $k$ be a field. Is it true that every $k$-endomorphism of
$k[x_1,...,x_n]$ taking every coordinate  to a coordinate is  an automorphism?}

\

\noindent In \cite{ess} the question was answered by van den Essen and
Shpilrain themselves in the positive
for an arbitrary field $k$ when $n=2$.
The question  was solved by Jelonek  \cite{essen} affirmatively for algebraically
closed fields $k$ of characteristic zero for all $n$ by geometric method based on

\

\noindent {\bf Derksen's observation.} (see \cite{ess}) Let $k$
be an algebraically closed field. A $k$-endomorphism $\phi$
of $k[x_1,\dots,x_n]$ taking every $k$-linear coordinate of $k[x_1, \dots, x_n]$
to a coordinate of $k[x_1, \dots, x_n]$ must have nonzero constant Jacobian
$J(\phi)$ in $k$.

\

\noindent For the related {\it linear coordinate preserving problem} for polynomial
algebras,
see Mikhalev, Yu and Zolotykh \cite{myz}, Cheng and ven den Essen  \cite{ec},
and Gong and Yu \cite{gy1}.
For another related {\it automorphic orbit problem} for polynomial algebras, see
van den Essen and Shpilrain \cite{ess},  Yu \cite{yu},  Gong and Yu \cite{gy},
and Li and Yu \cite{ly}.

\

\noindent The purpose of this note is to prove the following two new results.

\begin{theorem} \label{fixingvariable}
Let $R:=k[x_{n+1},\dots,x_{n+m}]$ where $k$ is an infinite field, $m>0$.
Let $\phi:=(f_1,\dots,f_n)$ be an $R$-endomorphism
of $A:=R[x_1,\dots,x_n]$ taking every
$R$-linear coordinate of $A$ to an $R$-coordinate of $A$.
Then $$J(\phi)=J_{x_1,\dots,x_n}(f_1,\dots,f_{n}):=\det[({f_i})^{\prime}_{x_j}]\in k^*.$$
\end{theorem}

\

\noindent Note if we replace $R$ by a field $k$,
the statement of Theorem \ref{fixingvariable} is generally not true,
unless $k$ is algebraically closed
(Derksen's observation). For non-algebraically closed
fields $k$, see  Mikhalev, J.-T.Yu and Zolotykh \cite{myz}, and Gong and Yu \cite{gy1} for counterexamples.

\begin{theorem}\label{jacobian}
Let $\phi:=(f_1,\dots,f_n)$ be a $k$-endomorphism of

\noindent $k[x_1,\dots,x_n]$ over a field $k$ of
characteristic zero taking every tame        \  \  \  \  \ coordinate of
$k[x_1,\dots,x_n]$ to a coordinate of $k[x_1,\dots,x_n]$.
Then $$J(\phi)=J_{x_1,\dots,x_n}(f_1,\dots,f_n):=\det[({f_i})^{\prime}_{x_j}]\in k^*.$$
\end{theorem}

\

\noindent In the sequel $k$ always denotes a field with a fixed algebraic closure $K$.
Endomorphisms (automorphisms) always means $k$-endomorphisms ($k$-automorphisms)
unless otherwise specified.

\

\section{\bf Preliminaries}

\noindent Recall
that an automorphism of $k[x_1,\cdots,x_n]$ is tame if it can be decomposed to  product of  linear  and elementary automorphisms, and a coordinate $p$ (i.e. a component
of an automorphism) is called tame if $p$ is a component of a tame automorphism. For an endomorphism
$\phi:=(f_1,\dots,f_n)$ of $k[x_1,\dots,x_n]$,
we use $$J(\phi):=J(f_1,\dots,f_n):=J_{x_1,\dots,x_n}(f_1,\dots,f_n)$$ to denote
its Jacobian $\det[({f_i})^{\prime}_{x_j}]$.

\

\noindent Let  $\phi:=(f_1,\dots,f_n)$ be an endomorphism of $k[x_1,\dots,x_n]$
taking a coordinate $p$ of $k[x_1,\dots,x_n]$ to a coordinate
of $k[x_1,\dots,x_n]$ and let $\sigma:=(g_1,\dots,g_n)$ be any automorphism
of $k[x_1,\dots,x_n]$. Obviously, $$\sigma\circ\phi:=(f(g_1,\dots,g_n),\dots,f_n(g_1,\dots,g_n))$$ is also an endomorphism of $k[x_1,\dots,x_n]$ taking the same
coordinate $p$ to a coordinate. Moreover, $J(\phi)\in k^*$  if and only if $J(\sigma\circ\phi)\in k^*$.

\

\noindent We need the following two lemmas.

\begin{lemma}\label{infinite}
Let $f_1,\dots,f_{n-1}\in k[x_1,\dots,x_n]$
over an infinite field $k$
such that
$$\deg_{x_n}J(f_1,\dots,f_{n-1},x_n)>0.$$
Then there exist $a_1,\dots,a_{n-1}\in k$,\ $b\in K$ and
$h_1,h_2,\dots,h_{n-1}\in k[x_n]$\  \  \
(without loss of
 generality we may assume $h_1=1$ after acting a transposition
 on $\{f_1,\dots,f_{n-1}\}$ )
such that the gradient (partial derivatives) of
$$g:=f_1+h_2(x_n)f_2+\dots+h_{n-1}(x_n)f_{n-1}$$
with respect to $(x_1, \dots, x_{n-1})$
is $(0,\dots,0)$ at the point
$$P=(a_1,\dots,a_{n-1},b).$$
\end{lemma}
\begin{proof}
Let
$$G(x_1,\dots,x_n):=J(f_1,\dots,f_{n-1},x_n)=\sum_{i=0}^mp_i(x_1,\cdots,x_{n-1})x_n^i$$ where $m\geq 1$ and $p_m(x_1,\cdots,x_{n-1})\not=0$. Since $k$ is infinite,
we may choose $a_1,\cdots,a_{n-1}\in k$ such that $p_m(a_1,\cdots, a_{n-1})\in k^*$. Then $G(a_1,\cdots,a_{n-1},x_n)\in k[x_n]-k$ and there exists some $b\in K$ such that
$$G(a_1,\cdots,a_{n-1},b)=0.$$
Hence $E=k(b)=k[b]$ is a finite algebraic extension of $k$. Let $P=(a_1,\cdots, a_{n-1}, b)$ be a point in $E^n$ and for each $f(x_1,\dots,x_n)\in k[x_1,\dots,x_n]$, define
$$f(x_1,\dots,x_n)\vert_P:=f(a_1,\cdots, a_{n-1},b)$$
We have the first $(n-1)$ rows of the determinant
$J(f_1,\dots,f_{n-1},x_n)\vert_P$ are
$k[b]$-linearly dependent. Therefore, there exist
$$h_1(x), h_2(x), \cdots, h_{n-1}(x)\in k[x]$$ such that
$$(h_1(b)f_1+h_2(b)f_2+\cdots+h_{n-1}(b)f_{n-1})^{\prime}_{x_i}\vert_P=0$$
for all $i=1,\cdots, n-1$ and not all  $h_1(b),\cdots, h_{n-1}(b)$ are  zero.
Without loss of generality, we may assume $h_1(b)\not=0$, replace $h_i(b)$ by $h_1^{-1}(b)h_i(b)$ for all $i$, we may assume $h_1(x)=1$.
Now define $$g:=f_1+h_2(x_n)f_2+\cdots+h_{n-1}(x_n)f_{n-1}\in k[x_1,\dots,x_n].$$ It is easy to see that
$$g^{\prime}_{x_i}\vert_P=(f_1+h_2(b)f_2+\cdots+h_{n-1}(b)f_{n-1})^{\prime}_{x_i}\vert_P=0,\ \ \ \
\forall i=1,\cdots,n-1.$$
\qed

\

\begin{lemma} \label{zerochar}
Let $f_1,\dots,f_{n-1}\in k[x_1,\dots,x_n]$
over a field $k$ of characteristic zero
such that $\deg_{x_n}J(f_1,\dots,f_n,x_n)>0$.
Then there exist $a_1,\dots,a_{n-1}\in k$,\ $b\in K$ and
$h_1,h_2,\dots,h_{n-1},h_n\in k[x_n]$
(without loss of generality, we may assume $h_1=1$ after acting
a transposition on $\{f_1, \dots,f_{n-1}\}$)
such that the gradient (partial derivatives) of
$$u:=f_1+h_2(x_n)f_2+\dots+\dots+h_{n-1}(x_n)f_{n-1}+h_n$$
with respect to $(x_1,\dots,x_{n-1},x_n)$
is $(0, \dots, 0, 0)$ at the point
$$P=(a_1, \dots, a_{n-1}, b).$$
\end{lemma}
\begin{proof}
Using the same notations in the Lemma \ref{infinite},
$u=g+h_n$, where $h_n=h_n(x_n)$ is to be determined.
Define $v(b):=g^{\prime}_{x_n}\vert_P\in k[b]$
for some $v(x)=c_0+c_1x+\dots+c_sx^s\in k[x]$.
Define (here we need $k$ to be characteristic zero)
$$h_n(x_n):=-c_0x_n-(1/2)c_1x_n^2-\dots-(1/(s+1))c_sx_n^{s+1}\in k[x_n].$$
It is easy to see that
$$u^{\prime}_{x_i}\vert_P=g^{\prime}_{x_i}\vert_P+h_n(x_n)^{\prime}_{x_i}=0+0=0,\ \ \ \forall i=1,\cdots, n-1$$
and $u^{\prime}_{x_n}\vert_P=v(b)-v(b)=0.$
\end{proof}

\

\section{\bf Proofs of the main results}

\noindent {\bf Proof of Theorem \ref{fixingvariable}.} For simplicity we only present the proof for
$m=1$, the general case can be proved similarly, by an enhanced version of Lemma \ref{infinite}.
Suppose on the contrary, $J(\phi)$ is not constant.
If $x_{n+1}$ does not appear in $J(\phi)$.
We may assume $\deg_{x_1}J(\phi)$
is the highest among all $\deg_{x_i}J(\phi)$.
Replace $\phi$ by $\sigma\circ\phi$, where
$\sigma:=(x_1+x_{n+1},x_2,\dots,x_n)$. So we
may assume that $x_{n+1}$ appears in $J(\phi)$. By Lemma \ref{infinite}
there exist $a_1,\dots, a_n\in k$,\ $b\in K$ and
$h_1, h_2,\dots,h_n\in R[x_{n+1}]$
(Without loss of generality we may assume $h_1=1$)
such that $$g:=f_1+h_2(x_{n+1})f_2+\dots+h_n(x_{n+1})f_n$$
with gradient (partial derivatives) with respect to
$(x_1,\dots,x_n)$ is $(0,\dots,0)$
at the point
$$P=(a_1,\dots,a_n, b),$$ so $g$
cannot be an $R$-coordinate of $A$. On the other hand, as
$$p:=x_1+h_2(x_{n+1})x_2+\dots+h_n(x_{n+1})x_n$$ is an $R$-linear
coordinate of $A$
with a corresponding $R$-automorphism
$(p, x_2,\dots,x_n)$ of $A$,\ $g=\phi(p)$ is also an $R$-coordinate of $A$. A contradiction.
\qed

\

\noindent {\bf Proof of Theorem \ref{jacobian}.} We may assume $f_n=x_n$, otherwise replace $\phi$
by $\psi\circ\phi$, where $\psi$ is an automorphism
taking the coordinate $f_n$ to $x_n$.
Suppose on the contrary, $J(\phi)$ is not nonzero constant.
If $x_{n}$ does not appear in $J(\phi)$,
we may assume that $\deg_{x_1}J(\phi)$
is the highest among all $\deg_{x_i}J(\phi)$.
Replace $\phi$ by $\sigma\circ\phi$, where
$\sigma:=(x_1+x_{n},x_2,\dots,x_n)$. So we
may assume that $x_{n}$ appears in $J(\phi)$.
By Lemma \ref{zerochar},
there exist $a_1,\dots,a_{n-1}\in k$,\ $b\in K$ and
$h_1,h_2,\dots,h_{n-1},h_n\in k[x_n]$
(without loss of generality we may assume $h_1=1$)
such that the partial derivatives of
$$u:=f_1+h_2(x_n)f_2+\dots+\dots+h_{n-1}(x_n)f_{n-1}+h_n$$
with respect to $(x_1,\dots,x_n)$
is $(0,\dots,0)$ at the point
$$P=(a_1,\dots,a_{n-1},b),$$
hence $u$ cannot be a coordinate of $k[x_1,\dots,x_n]$.
On the other hand, as $$q:=x_1+h_2(x_n)x_2+\dots+h_{n-1}(x_n)x_{n-1}+h_n(x_n)\in k[x_1, \dots, x_n]$$
is a tame coordinate of $k[x_1, \dots, x_n]$ with a corresponding elementary
automorphism
$$(q, x_2,\dots,x_{n-1},x_n),$$ of $k[x_1,\dots,x_n]$,\
$u=\phi(q)$ is also a coordinate. A contradiction.
\end{proof}

\

\noindent {\bf Acknowledgement.}  The authors would like to thank Z.Jelonek
for kindly showing  his very recent result \cite{J} where
Question 1 in this note is solved affirmatively for characteristic zero
and arbitrary $n$ by geometric method
based on Theorem \ref{jacobian} in this note.

\


\begin{thebibliography}{99}

\bibitem{ec} Cheng, C. C.-A; van den Essen, A,
Endomorphisms of the plane sending linear coordinates to coordinates,
{\bf Proc. Amer. Math. Soc. 128} (2000) 1911-1915.


\bibitem{ess} Van den Essen, A; Shpilrain, V, Some
combinatorial questions about polynomial mappings, {\bf J.Pure Appl.Algebra 109}
(1997) 47-52.

 \bibitem {gy1}
 Gong, S.-J; Yu, J.-T, The linear coordinate preserving problem,
 {\bf Comm. Algebra 36} (2008) 1354-1364.

\bibitem{gy} Gong, S.-J; Yu, J.-T,
Test elements, retracts and automorphic orbits,
{\bf J.Algebra  320} (2008) 3062-3068.


\bibitem{essen} Jelonek, Z, A solution of the problem of van den Essen and Shpilrain, {\bf
J. Pure Appl. Algebra 137},  (1999) 49-55.



\bibitem{J} Jelonek, Z, A solution of a question of van den Essen and Shpilrain,
preprint 2011.

\bibitem{ly} Li, Y.-C; Yu, J.-T,
Applications of degree estimate for subalgebras,
{\bf C. R. Acad. Bulgare Sci. 64} (2011) 165-172.



\bibitem{myz} Mikhalev, A. A; Yu, J.-T; Zolotykh, A. A, Images of coordinate polynomials,
 {\bf Algebra Colloq. 4} (1997) 159-162.



\bibitem{yu} Yu, J.-T, Automorphic orbit problem for polynomial algebras,
{\bf J.Algebra 319} (2008) 966-970.



\end{thebibliography}
\end{document}